\documentclass[leqno]{amsart}
\pdfoutput=1
\usepackage[whole]{bxcjkjatype}
\usepackage[a4paper, margin=3cm]{geometry}
\usepackage{amsthm}
\usepackage{amsmath,amssymb,mathtools,mathdots}
\usepackage{mathrsfs}
\usepackage{url}
\usepackage[all]{xy}
\usepackage{tikz}
\usepackage{tikz-cd}
\usetikzlibrary {arrows,backgrounds}
\usetikzlibrary{shapes.geometric, intersections, calc}
\usetikzlibrary{decorations.pathmorphing, arrows}
\usepackage{comment}
\usepackage{placeins}

\usepackage{multirow}
\usepackage{diagbox}
\usepackage{makecell}
\usepackage{caption}

\newtheorem{theorem}{Theorem}[section]
\newtheorem*{theorem*}{Theorem}
\newtheorem{lemma}[theorem]{Lemma}
\newtheorem*{lemma*}{Lemma}

\newtheorem*{proposition*}{Proposition}

\newtheorem*{corollary*}{Corollary}

\newtheorem*{claim*}{Claim}

\newtheorem*{fact*}{Fact}
\newtheorem{conjecture}[theorem]{Conjecture}
\newtheorem*{conjecture*}{Conjecture}
\theoremstyle{definition}

\newtheorem*{definition*}{Definition}

\newtheorem*{example*}{Example}
\newtheorem{remark}[theorem]{Remark}
\newtheorem*{remark*}{Remark}

\newtheorem*{question*}{Question}

\newtheorem*{assumption*}{Assumption}
\numberwithin{equation}{section}
\allowdisplaybreaks[1]

\DeclareMathOperator{\id}{id}

\DeclareMathOperator{\Span}{Span}

\newcommand{\C}{{\mathbb C}}
\newcommand{\R}{{\mathbb R}}

\newcommand{\Half}{{\mathbb H}}

\newcommand{\iu}{{\mathrm{i}}}

\newcommand{\conj}[1]{{\overline{#1}}}

\newcommand{\fg}{{\mathfrak g}}

\usepackage{multirow}

\address{Graduate School of Mathematical Sciences, 
The University of Tokyo,
3-8-1 Komaba, Meguro-ku, 
Tokyo 153-8914, Japan}
\email{shuho@ms.u-tokyo.ac.jp}
\begin{document}
\title
[
Non-Stein universal cover of a solvmanifold
]
{
A solvmanifold with a left-invariant complex structure 
whose universal cover is not Stein
}
\author{Shuho Kanda}
\date{}
\maketitle
\begin{abstract}
    We construct a simply connected solvable Lie group $G$ 
admitting lattices and a left-invariant complex structure $J$ such that 
$(G,J)$ is not Stein. 
This provides a counterexample to Hasegawa's conjecture on 
the Stein property of simply connected unimodular solvable Lie groups 
with left-invariant complex structures. 

\end{abstract}

\setcounter{tocdepth}{1}
\tableofcontents

\section{Introduction}

\subsection{Background}

Let $G$ be a simply connected solvable Lie group.
If $G$ admits a left-invariant complex structure $J$ and a lattice $\Gamma$,
then $J$ naturally descends to the compact manifold $\Gamma \backslash G$,
and $X=(\Gamma \backslash G, J)$ 
becomes a compact complex manifold.
Such a manifold $X$ is called a solvmanifold with left-invariant complex structure. 
Solvmanifolds have been actively studied in complex geometry, especially
in non-Kähler geometry, since they exhibit many unusual phenomena in
complex geometry.
When $G$ is nilpotent, $X$ is called a nilmanifold.
Although nilmanifolds are easier to study than general solvmanifolds,
many mysteries still remain.

In this paper, we focus on the properties of the universal cover of $X$.
In \cite{Has10}, 
while studying deformations of complex structures on
solvmanifolds, Hasegawa proposed the following conjecture.

\begin{conjecture}[\cite{Has10}]
\label{Conj:Hasegawa}
    Let $G$ be a $2n$-dimensional 
    simply connected unimodular solvable (nilpotent) Lie group. 
    Then for every left-invariant complex structure $J$ on $G$, 
    the complex manifold $(G,J)$ is Stein (biholomorphic to $\C^n$ respectively).
\end{conjecture}

Recently, the nilpotent case was settled in \cite{HRSTW26}. 
That is, if $G$ is nilpotent, then in the above setting one has
$(G,J) \simeq \C^n$.
The same conclusion does not hold in the solvable case.
For example, Oeljeklaus--Toma manifolds, which are higher-dimensional
generalizations of Inoue surfaces, are solvmanifolds with left-invariant
complex structures, but their universal covers are of the form
$\Half^s \times \C^t$, 
where $\Half$ denotes the upper half plane. 
Thus the universal cover need not be $\C^n$.
Conjecture \ref{Conj:Hasegawa} asserts, however, 
that it should still be Stein.
Although some evidence supporting this conjecture had been obtained, 
for instance in \cite{AT26}, essentially no further progress had been made.

\subsection{Result}

In this paper, we give a counterexample to Conjecture \ref{Conj:Hasegawa}.
The counterexample is obtained by putting a different complex structure on
the Nakamura manifold \cite{Nak75}.
The Nakamura manifold is the quotient $X=\Gamma\backslash G$ by a lattice
$\Gamma$, where $G$ is, up to isomorphism, the unique non-nilpotent type
among $3$-dimensional simply connected unimodular complex solvable Lie
groups:
\begin{equation}\label{Equ:Nakamura}
G=\C \ltimes_{\rho} \C^2, \quad \rho(z)= \mathrm{diag}(e^z,e^{-z}).
\end{equation}
With its original complex structure, the universal cover $G$ is of course
biholomorphic to $\C^3$.
However, we regard $G$ as a $6$-dimensional real Lie group and equip it
with another left-invariant complex structure. 
The precise statement and the explicit description of 
the complex structure are as follows.

\begin{theorem}
Let
\[
G=\C \ltimes_{\rho} \C^2, \quad \rho(z)= \mathrm{diag}(e^z,e^{-z})
\]
be the simply connected complex solvable Lie group, and let
$\fg=\C \ltimes_{d\rho} \C^2$ be its Lie algebra. We take an
$\R$-basis of $\fg$ as
\begin{align*}
    T=(1,0,0), \quad X=(0,1,0), \quad Y=(0,0,1), \\
    T'=(\iu,0,0), \quad X'=(0,\iu,0), \quad Y'=(0,0,\iu). 
\end{align*}
Define a complex structure $J$ on $\fg$ determined by the complex
Lie subalgebra $\fg^{1,0} \subset \fg_{\C}$ given by
\[
\fg^{1,0}=\Span_{\C}({T-\iu T', X-\iu X'+2\iu(X'+T'), Y-\iu Y'}).
\]
Then the complex manifold $G$ endowed with the left-invariant complex
structure induced by $J$ is not Stein.
\end{theorem}

The idea of the construction is as follows.
In \cite{Sno86}, 
Snow constructed a holomorphic map
\[
\Phi \, \colon \, (G,J) \to G_{\C} / G^{0,1}
\]
for a left-invariant complex structure $J$ on $G$.
Here $G_{\C}$ and $G^{0,1}$ are the Lie groups corresponding to
$\fg \otimes_{\R} \C$ and to the $(-\iu)$-eigenspace
$\fg^{0,1}$ of $J$, respectively.
When $G$ is simply connected and solvable, one can show that
$G_{\C} / G^{0,1} \simeq \C^n$ and that
$\Phi$ is a covering map onto its image.
It follows from these properties that $(G,J)$ is the universal cover of
$D=\Phi(G) \subset G_{\C} / G^{0,1} \simeq \C^n$.

When $G$ is equipped with a complex conjugation, we have
\[
G_{\C} \simeq G \times G
\]
and, under this isomorphism, the embedding $G \hookrightarrow G_{\C}$ is
given by $x \mapsto (x,\conj{x})$.
Thus the problem becomes whether one can construct
$\fg^{0,1} \subset \fg_{\C} \simeq \fg \oplus \fg$ in such a way that the
universal cover of
\[
D=
\{[(x,\conj{x})] \in G_{\C} / G^{0,1} \mid x \in G\} 
\subset G_{\C} / G^{0,1} \simeq \C^n
\]
is not Stein.
In this paper, by taking $G$ to be the group in (\ref{Equ:Nakamura}), we
construct such an example for which
$D \simeq (\C^2 \backslash \R^2) \times \C$.
This domain is not Stein, and we also show that its universal cover is not
Stein.

\addtocontents{toc}{\protect\setcounter{tocdepth}{0}}
\section*{Acknowledgments}
\addtocontents{toc}{\protect\setcounter{tocdepth}{1}}
%目次に入れないようにしてある 
The author is grateful to Hisashi Kasuya and Keizo Hasegawa for helpful
discussions on this work. 
This research is supported by JSPS KAKENHI Grant number 24KJ0931. 

\section{Preliminaries}
\label{section:Preliminaries}

\subsection{Complexification of a complex Lie algebra}

Let $\fg_0$ be a real Lie algebra, and let
$\fg = (\fg_0)_{\mathbb C}=\fg_0\otimes_{\R}\mathbb C$
be its complexification. 
We regard $\fg$ as a real Lie algebra endowed
with the complex structure $J_{st}$, 
and consider again its complexification
$\fg_{\mathbb C}$. Define
\[
\iota_{\pm} \colon \fg \to \fg_{\mathbb C}, \qquad
X \mapsto \frac{1}{2}(X \mp \iu J_{st}X).
\]
Then $\iota_+$ is complex linear, while $\iota_-$ is complex anti-linear.
Therefore, by composing the latter with the complex conjugation
$\sigma \colon \fg \to \fg$, 
we obtain a complex linear map
\begin{equation}\label{Equ:iota}
\iota = \iota_+ \oplus (\iota_- \circ \sigma)
\colon \fg \oplus \fg \to \fg_{\mathbb C}.
\end{equation}
One checks that this gives an isomorphism of complex Lie algebras. Its
inverse is given by
\[
U+\iu V \mapsto (U+J_{st}V,\overline{U-J_{st}V}),
\qquad U,V\in \fg.
\]
In particular, under the identification $\fg_{\mathbb C}\simeq
\fg\oplus\fg$ via $\iota$, the natural inclusion
\[
\fg \hookrightarrow \fg_{\mathbb C}, \qquad X\mapsto X\otimes 1,
\]
is expressed as
\[
X\mapsto (X,\overline X), \qquad X\in\fg.
\]

On the Lie group level, let $G$ and $G_{\mathbb C}$ be the simply
connected Lie groups corresponding to $\fg$ and $\fg_{\mathbb C}$,
respectively. 
Then there is an isomorphism of complex Lie groups
\begin{equation}\label{Equ:GG}
G_{\mathbb C}\simeq G\times G,
\end{equation}
under which the natural inclusion $G\hookrightarrow G_{\mathbb C}$ is
expressed as
\[
x\mapsto (x,\overline x), \qquad x\in G.
\]

\subsection{The Snow map}

Here we explain the properties of the map
$\Phi \, \colon \, (G,J) \to G_{\C} / G^{0,1}$
constructed by Snow in \cite{Sno86}.
This map also plays an essential role 
in the proof of Hasegawa's conjecture
in the nilpotent case \cite{HRSTW26}. 
Here we review the material presented 
in Subsection 1.3 of \cite{Sno86}. 

Let $G$ be a simply connected solvable Lie group endowed with a
left-invariant complex structure $J$, and let $G_{\C}$ be its
complexification.
These correspond to the Lie algebra $\fg$ with $J$ and its complexification
$\fg_{\C}$.
The complex structure $J$ determines a complex Lie subalgebra
$\fg^{0,1} \subset \fg_{\C}$.
By a classical Malcev's theorem 
(see, for example, \cite[Theorem~3.18.12]{Var84}), 
the Lie group $G^{0,1} \subset G_{\C}$
corresponding to $\fg^{0,1}$ is closed. 

\begin{theorem}\label{Thm:Snow}
    The composition
    \[
    \Phi \, \colon \, (G,J) \hookrightarrow G_{\C} 
    \twoheadrightarrow G_{\C}/G^{0,1}
    \]
    is a local biholomorphism.
    Moreover, $\Phi$ is a covering map onto its image $D=\Phi(G)$.
    In particular, $(G,J)$ is the universal covering of $D$.
\end{theorem}

\begin{proof}
    The differential at the identity
    $d\Phi_e \, \colon \, \fg \to \fg_{\C} / \fg^{0,1}$
    is a linear isomorphism.
    Since $\Phi$ is equivariant with respect to the left action of $G$,
    it is a local diffeomorphism.
    Moreover, by the left-invariance of $J$, the map $\Phi$ is holomorphic.

    Now set $K = G \cap G^{0,1}$. Then $\Phi$ factors through the
    projection as
    \[
    G \twoheadrightarrow G/K \overset{\sim}{\longrightarrow} D 
    \hookrightarrow G_{\C}/G^{0,1}.
    \]
    Since $K$ is discrete by $\fg \cap \fg^{0,1}=\{0\}$,
    the map $G \twoheadrightarrow G/K$ is a covering map, and hence
    $\Phi$ is also a covering map onto its image.
\end{proof}

\begin{remark}
    In this situation, one can show in general that
    $G_{\C}/G^{0,1} \simeq \C^n$ \cite[p.174]{HO81}.
    In this paper, however, we do not use this general result, since we
    construct this isomorphism explicitly.
\end{remark}

According to this theorem, the complex manifold structure of $(G,J)$ can
be understood by looking at the image of $G$ in $G_{\C}/G^{0,1}$.
In the next section, we construct $G^{0,1}$ explicitly.

\section{Construction of the example}

We define a $3$-dimensional simply connected complex Lie group $G$ by
\[
G=\C \ltimes_{\rho} \C^2, \quad \rho(z)= \mathrm{diag}(e^z,e^{-z}). 
\]
This is the complexification of $G_0=\R \ltimes_{\rho} \R^2$, 
and $G$ admits a lattice by [Nak75]. 
We take a
$\C$-basis of $\fg= \C \ltimes_{d\rho}\C^2$ as
\[
    T=(1,0,0), \quad X=(0,1,0), \quad Y=(0,0,1). 
\]
The brackets are given by $[T,X]=X$ and $[T,Y]=-Y$.

By the isomorphism (\ref{Equ:GG}), from now on we write $G_{\C}$ for $G \times G$.
A point of $G_{\C}$ is described by the $6$ complex coordinates $((t_1,x_1,y_1),(t_2,x_2,y_2))$.
Corresponding to these coordinates, 
we write $\{T_1,X_1,Y_1,T_2,X_2,Y_2\}$ 
for the $\C$-basis of $\fg_{\C}=\fg\oplus\fg$ obtained from the two copies
of the basis of $\fg$. 

Now define complex subalgebras $\fg^{1,0}$ and $\fg^{0,1}$ of $\fg_{\C}$ by
\begin{align}
\fg^{1,0} &= \Span_{\C}({T_1, T_2+X_2, Y_1}) \label{Equ:g^10} \\
\fg^{0,1} &= \Span_{\C}({T_1+X_1, T_2, Y_2}) \notag
\end{align}
This defines a complex structure $J$ on $\fg$.
Since one checks that $\fg^{1,0}$ and $\fg^{0,1}$ are subalgebras, 
the complex structure $J$ is integrable.

The connected Lie subgroup of $G$ whose Lie algebra is
$\C(T+X)\subset \fg$ is
\[
\left\{ (s,e^s-1,0) \in G 
    \mid s \in \C \right\}.
\]
Similarly, the connected Lie subgroup corresponding to
$\Span_{\C}(\{T,Y\})\subset \fg$ is
\[
    \left\{ (t,0,r) \in G 
    \mid t,r \in \C \right\}.
\]
Therefore, the complex Lie group $G^{0,1}$ corresponding to
$\fg^{0,1}$ is given by
\[
G^{0,1} = \left\{ ((s,e^s-1,0),(t,0,r)) \in G_{\C} 
\mid s,t,r \in \C \right\}.
\]

\begin{lemma}
    The map
    \[
    Q \,\colon\, G_{\C} \to \C^3, \quad 
    ((t_1,x_1,y_1),(t_2,x_2,y_2)) \mapsto (e^{t_1}-x_1,x_2,y_1)
    \]
    induces a biholomorphism
    \[
    \widetilde{Q} \, \colon\, G_{\C} / G^{0,1} 
    \overset{\sim}{\longrightarrow} \C^3.
    \]
\end{lemma}

\begin{proof}
    We first observe that $Q$ is invariant 
    under right multiplication by $G^{0,1}$.
    Indeed,
    \begin{align*}
    &Q(((t_1,x_1,y_1),(t_2,x_2,y_2)) \cdot ((s,e^s-1,0),(t,0,r))) \\
    =&Q((t_1+s, e^{t_1}(e^s-1)+x_1,y_1),(t_2+t,x_2,e^{-t_2}r+y_2))) \\
    =&(e^{t_1+s}-(e^{t_1}(e^s-1)+x_1), x_2, y_1) \\
    =&(e^{t_1}-x_1,x_2,y_1) \\
    =&Q(((t_1,x_1,y_1),(t_2,x_2,y_2))).
    \end{align*}
    Hence $Q$ induces a holomorphic map $\widetilde{Q}$.

    Define
    \[
    P \, \colon \, 
    \C^3 \to G_{\C} / G^{0,1}, \quad (a,b,c) \mapsto [((0,1-a,c),(0,b,0))].
    \]
    Then $\widetilde{Q} \circ P = \id$ is clear. Moreover,
    \begin{align*}
    &(P\circ \widetilde{Q}) ([((t_1,x_1,y_1),(t_2,x_2,y_2))]) \\
    =&[((0,1-(e^{t_1}-x_1),y_1),(0,x_2,0))] \\
    =&[((t_1,x_1,y_1),(t_2,x_2,y_2)) \cdot
    ((-t_1,e^{-t_1}-1,0),(-t_2,0,-e^{t_2}y_2))]\\
    =&[((t_1,x_1,y_1),(t_2,x_2,y_2))].
    \end{align*}
    Thus $P \circ \widetilde{Q}= \id$ as well, and hence
    $\widetilde{Q}$ is bijective.
\end{proof}

By composing the Snow map $\Phi$ defined in Section \ref{section:Preliminaries}
with $\widetilde{Q}$, we obtain a holomorphic map
\[
\Psi \, \colon \, (G,J) \hookrightarrow G_{\C} \twoheadrightarrow 
G_{\C} / G^{0,1} \overset{\sim}{\longrightarrow} \C^3.
\]
More explicitly, it is given by
\[
\Psi \, \colon \, (t,x,y) \mapsto ((t,x,y),(\conj{t},\conj{x},\conj{y}))
\mapsto [(t,x,y),(\conj{t},\conj{x},\conj{y})] 
\mapsto (e^t-x,\conj{x},y).
\]
The image $D=\Psi(G)$ is
\[
D = \left\{ (\xi_1,\xi_2,\xi_3) \in \C^3 \mid \xi_1+\conj{\xi_2} \neq 0 \right\}.
\]

\begin{lemma}
    The universal cover $\widetilde{D}$ of $D$ is not Stein.
\end{lemma}

\begin{proof}
The linear transformation
\[
\C^3 \overset{\sim}{\longrightarrow} \C^3, 
\quad (\xi_1,\xi_2,\xi_3) \mapsto
(\xi_1-\xi_2, \iu(\xi_1+\xi_2),\xi_3)
\]
identifies $D$ with $(\C^2 \backslash \R^2) \times \C$.
Thus it is enough to prove the assertion for
$D'=(\C^2 \backslash \R^2) \times \C$.
Let $p \, \colon \, \widetilde{D'} \to D'$ 
be the universal covering. We show that $\widetilde{D'}$ is not Stein.

We write $D'$ as the tube domain
\[
D' = T_{\Omega} \coloneqq \R^3 + \iu\Omega, \quad 
\Omega=\{(x,y,z) \in \R^3 \mid (x,y) \neq (0,0) \}. 
\]
Take two points
\[
P=(0,\iu,0), \quad Q=(0,-\iu,0)
\]
in $D'$, and fix a point $\tilde{P} \in \widetilde{D'}$ such that
$p(\tilde{P})=P$.
Set
\[
\Omega^{\pm}= \R^3 \backslash \{(x,0,z) \in \R^3 \mid \pm x \le 0 \}.
\]
Then $\Omega^{\pm}$ are simply connected and
$\Omega = \Omega^+ \cup \Omega^-$. 
For each sign, lift $T_{\Omega^{\pm}}$ to the component of
$p^{-1}(T_{\Omega^{\pm}})$ containing $\tilde{P}$, and denote it by
$\widetilde{T_{\Omega^{\pm}}}$.
With these choices, the point $Q$ has two distinct lifts
$\tilde{Q}^+,\tilde{Q}^-$.

Now let $f \in \mathcal{O}(\widetilde{D'})$.
Restricting $f$ to $\widetilde{T_{\Omega^{\pm}}}$ and descending it by
$p$, we obtain holomorphic functions
$f^{\pm} \in \mathcal{O}(T_{\Omega^{\pm}})$. 
By Bochner's tube theorem, 
each $f^{\pm}$ extends
holomorphically to $\C^3$, 
which is the convex hull of $T_{\Omega^{\pm}}$.
Since $f^+$ and $f^-$ agree in a neighborhood of $P$, the identity theorem
implies that their extensions agree on all of $\C^3$.
Therefore
\[
f(\tilde{Q}^+)=f^+(Q)=f^-(Q)=f(\tilde{Q}^-).
\]
Thus the two distinct points $\tilde{Q}^{\pm}$ of $\widetilde{D'}$ cannot
be separated by holomorphic functions. Hence $\widetilde{D'}$ is not
Stein.
\end{proof}

By Theorem \ref{Thm:Snow}, $(G,J)$ is biholomorphic to
$\widetilde{D}$.
It follows that $(G,J)$ is not Stein. 

\begin{remark}
    In \cite{Has10}, 
    it is concluded that, 
    for the Lie group $G$ defined by (\ref{Equ:Nakamura}), 
    every left-invariant complex structure $J$ on $G$ 
    makes $(G,J)$ biholomorphic to $\C^3$. 
    The example constructed here appears to
    provide a counterexample to this conclusion.
    More precisely, 
    it seems to correspond to a case which is not covered 
    in the analysis of the possible structures of \(J\) in \cite{Has10}. 
\end{remark}

Now, although we have explicitly constructed $(G,J)$, 
the structure of $J$ was given 
under the identification $\fg_{\C}=\fg \oplus \fg$. 
We describe this explicitly in terms of $\fg_{\C}=\fg \otimes \C$.

The complex Lie algebra $\fg$ has a $\C$-basis ${T,X,Y}$, 
while as an $\R$-basis one may take ${T,X,Y,T',X',Y'}$. 
Here $T',X',Y'$ denote the $\iu$-multiples of $T,X,Y$,
respectively. 
Let $J_{st}$ denote the natural complex structure on 
$\fg$ viewed as a complex Lie algebra. 
Then the $\iu$-eigenspace of $J_{st}$ in $\fg_{\C}$ is
\[
\fg^{1,0}_{J_{st}} =
\Span_{\C}({T-\iu T', X-\iu X', Y-\iu Y'}). 
\]
Now, 
by applying (\ref{Equ:iota}) to (\ref{Equ:g^10}), 
which was given under the identification $\fg_{\C}=\fg \oplus \fg$,  
we obtain
\begin{align*}
\fg^{1,0}_{J} &=
\Span{\C}({T-\iu T', T+X+\iu(T'+X'), Y-\iu Y'}) \\
&=
\Span_{\C}
({T-\iu T', X-\iu X'+2\iu(X'+T'), Y-\iu Y'}). 
\end{align*}

Putting the above together, 
we obtain the following theorem.
\begin{theorem}
Let $\fg=\C \ltimes_{d\rho} \C^2$.
Define a complex structure $J$ on $\fg$ determined by the complex
Lie subalgebra $\fg^{1,0} \subset \fg_{\C}$ given by
\[
\fg^{1,0}=\Span_{\C}({T-\iu T', X-\iu X'+2\iu(X'+T'), Y-\iu Y'}).
\]
Then the complex manifold $G$ endowed with the left-invariant complex
structure induced by $J$ is not Stein.
\end{theorem}

\bibliographystyle{amsalpha5}
\bibliography{CounterEx_Hasegawa}
\end{document}